\documentclass[twoside]{aiml16}

\usepackage{aiml16macro}

\usepackage{graphicx}
\usepackage{amsmath}
\usepackage{amssymb}
\usepackage{stmaryrd}
\usepackage{enumerate}

\newcommand*{\di}{\Diamond}
\newcommand*{\bo}{\square}
\newcommand*{\F}{\mathcal{F}}
\newcommand*{\M}{\mathcal{M}}

\newcommand*{\A}{\mathcal{A}}
\renewcommand*{\AA}{\mathbb{A}}
\newcommand*{\B}{\mathcal{B}}
\newcommand*{\LL}{\mathcal{L}}
\newcommand*{\BB}{\mathbb{B}}
\newcommand*{\MM}{\mathrm{Mod}}
\newcommand*{\NN}{\mathbb{N}}

\newcommand*{\VV}{\mathbb{V}}
\newcommand*{\EE}{\mathbb{E}}
\newcommand*{\G}{\mathcal{G}}

\newcommand*{\EF}[1]{\mathrm{FS}_{#1}}
\newcommand*{\ML}{\mathrm{ML}}
\newcommand*{\FO}{\mathrm{FO}}
\newcommand*{\uTL}{\mathrm{unary\text{-}TL}}
\newcommand*{\PAL}{\mathrm{PAL}}
\newcommand*{\EL}{\mathrm{EL}}

\renewcommand*{\phi}{\varphi}
\renewcommand*{\theta}{\vartheta}

\newcommand*{\tower}{\mathrm{twr}}
\newcommand*{\dom}{\mathrm{dom}}
\newcommand*{\powerset}{\mathcal{P}}
\newcommand*{\kaikki}{\square}
\newcommand*{\seur}[2]{\bo(#1,#2)}
\newcommand*{\s}{\mathrm{s}}
\newcommand*{\ms}{\mathrm{ms}}
\newcommand*{\cs}{\mathrm{cs}}
\newcommand*{\bisim}{\leftrightarroweq}
\newcommand*{\liitos}{\textstyle{\bigtriangleup}}
\newcommand*{\win}{(\mathrm{win})}
\newcommand*{\sep}{(\mathrm{sep})}
\newcommand*{\ordo}{\mathcal{O}}

\def\lastname{Hella and Vilander}

\begin{document}

\begin{frontmatter}
  \title{The Succinctness of First-order Logic over Modal Logic via a Formula Size Game}
  \author{Lauri Hella} \author{Miikka Vilander}
  \address{School of Information Sciences \\ University of Tampere}
  
\begin{abstract}
We propose a new version of formula size game for modal logic. The game characterizes
the equivalence of pointed Kripke-models up to formulas of given numbers of
modal operators and binary connectives. Our 
game is similar to the well-known Adler-Immerman game. However, due to a crucial difference
in the definition of positions of the game, its winning condition is simpler, and the 
second player (duplicator) does not have a trivial optimal strategy. Thus, unlike
the Adler-Immerman game, our game is a genuine two-person game.
We illustrate the use of the game by proving a nonelementary succinctness gap
between bisimulation invariant first-order logic $\FO$ and (basic) modal logic $\ML$. 
\end{abstract}

 \end{frontmatter}


\section{Introduction}

Succinctness is an important research topic that has been quite active in modal logic for the
last couple of decades; see, e.g., 
\cite{fo2,lutzsattlerwolter,past,markey,immermangames,lutz,figueira} for earlier work on this topic and \cite{epistemic,multimodalsuccinctness,frenchiliev,iliev,unionsuccinctness,syntaxtrees} 
for recent research.
If two logics $\LL$ and $\LL'$ have equal expressive power,  it is natural to ask, whether
there are properties that can be expressed in $\LL$ by a substantially shorter formula
than in $\LL'$ (or vice versa). For example, $\LL$ is \emph{exponentially more succinct} than
$\LL'$, if for every integer $n$ there is an $\LL$-formula 
$\phi_n$ of length $\ordo(n)$ such that any equivalent $\LL'$-formula $\psi_n$
is of length at least $2^n$.

Often such a gap in succinctness comes together with a similar gap in the complexity of 
the logics. For example, Etessami, Vardi and Wilke \cite{fo2} proved that, over $\omega$-words,
the two-variable fragment $\FO^2$ of first-order logic has the same expressive power 
as $\uTL$ (a weak version of temporal logic), but $\FO^2$ is exponentially more 
succinct than $\uTL$, and furthermore, the complexity of satisfiability for $\FO^2$
is NEXPTIME-complete, while the complexity of $\uTL$ is in NP \cite{sistla}.
However, succinctness does not always lead to a penalty in terms of complexity: 
an example is public announcement logic $\PAL$ which is exponentially more succinct
than epistemic logic $\EL$, but both have the same complexity,
as proved by Lutz in \cite{lutz}.

In order to prove succinctness results we need a method for proving lower bounds for the
length of formulas expressing given properties. The two most common methods 
used in the recent literature are  the \emph{formula size game} introduced by Adler and Immerman 
\cite{immermangames}, and \emph{extended syntax trees} due to Grohe and 
Schweikardt \cite{linearFO}. The latter was inspired by the former, and in fact, an extended
syntax tree is essentially a witness for the existence of a winning strategy in the
Adler-Immerman game. Thus, these two methods are equivalent, and the choice between
them is often a matter of convenience.

Originally, Adler and Immerman \cite{immermangames} formulated their game 
for the branching-time temporal logic $\mathrm{CTL}$. They used it for proving 
an $n!$ lower bound on the size of $\mathrm{CTL}$-formulas for expressing that 
there is a path on which each of the propositions $p_1,\ldots,p_n$ is true. As it is straightforward
to express this property by a formula of  $\mathrm{CTL}^+$ of size linear in $n$,
their result established that $\mathrm{CTL}^+$ is $n!$ times more succinct than 
$\mathrm{CTL}$, thus improving an earlier exponential succinctness result of Wilke 
\cite{ctl}.

After its introduction in \cite{immermangames}, the Adler-Immerman game, as well as 
the method of extended syntax trees, has been adapted to a host of modal languages.
These include epistemic logic \cite{epistemic}, multimodal logics with union and intersection
operators on modalities \cite{multimodalsuccinctness} and modal logic with contingency
operator \cite{syntaxtrees}, among others.

The Adler-Immerman game can be seen as a variation of the Ehrenfeucht-Fra\"{\i}ss{\'e} game,
or, in the case of modal logics, the bisimulation game. In the Adler-Immerman game, 
quantifier rank (or modal depth) is replaced by a parameter, usually called formula size, 
that is closely related to the length of the formula. Moreover, in order to use the game
for proving that a property is not definable by a formula of a given size, it is necessary
to play the game on pair $(\AA,\BB)$ of sets of structures instead of just a pair of single  
structures. 

The basic idea of the Adler-Immerman game is that one of the players, S (spoiler), 
tries to show that the sets $\AA$ and $\BB$ can be separated by
a formula of size $n$, while the other player, D (duplicator), aims to show that no 
formula of size at most $n$ suffices for this.
The moves that S makes in the game reflect directly the logical operators in a formula
that is supposed to separate the sets $\AA$ and $\BB$. Any pair $(\sigma,\delta)$ of 
strategies for the players S and D produces a finite game tree $T_{\sigma,\delta}$, and 
S wins this play if the size of $T_{\sigma,\delta}$ is at most $n$.
The strategy $\sigma$ is a winning strategy for S if using it, S wins every play of the game. 
If this is the case, then there is a formula of size at most $n$ that separates the sets, 
and this formula can actually be read from the strategy $\sigma$.

A peculiar feature of the Adler-Immerman game is that the second player, duplicator, 
can be completely eliminated from it. This is because D has an optimal strategy $\delta_{\max}$,
which is to always choose the maximal allowed answer; this strategy guarantees that 
the size of the tree $T_{\sigma,\delta}$ is as large as possible.  Thus, in this sense 
the Adler-Immerman game is not a genuine two-person game, but rather a one-person
game.

In the present paper, we propose another type of formula size game for modal logic.
Our game is a natural adaptation of the game introduced by Hella and V\"a\"an\"anen 
\cite{formulasize} for propositional logic and first-order logic. 
The basic setting in our game is the same as in the Adler-Immerman game: there are
two players, S and D, and two sets of structures that S claims can
be separated by a formula of some given size. The crucial difference is that in our
game we define positions to be tuples $(m,k,\AA,\BB)$ instead of just pairs $(\AA,\BB)$
of sets of structures, where $m$ and $k$ are parameters referring to the number of 
modal operators and binary connectives in a formula. In each move S has to decrease
at least one of the parameters $m$ or $k$. The game ends when the players reach a
position $(m^*,k^*,\AA^*,\BB^*)$ such that either there is a literal separating 
$\AA^*$ and $\BB^*$, or
$m^*=k^*=0$. In the former case, S wins the play; otherwise D wins.

Thus, in contrast to the Adler-Immerman game, to determine the winner in our game 
it suffices to consider a single ``leaf-node"
$(m^*,k^*,\AA^*,\BB^*)$ of the game tree. This also means that our game is a real two-person 
game: the final position $(m^*,k^*,\AA^*,\BB^*)$ of a play depends on the moves of D, and
there is no simple optimal strategy for D that could be used for eliminating the role of D
in the game. 

We believe that our game is more intuitive and thus, in some cases it may be easier to use than
the Adler-Immerman game. On the other hand, it should be remarked that the two games are 
essentially equivalent: The moves corresponding to connectives and modal operators
are the same in both games (when restricting to the sets $\AA$ and $\BB$ in a position
$(m,k,\AA,\BB)$). Hence, in principle, it is possible to translate a winning strategy in one
of the games to a corresponding winning strategy in the other. 

We illustrate the use of our game by proving a nonelementary succinctness gap
between first-order logic $\FO$ and (basic) modal logic $\ML$. More precisely, we define 
a bisimulation invariant property of pointed Kripke models by a first-oder formula 
of size $\ordo(2^n)$, and show that this property cannot be defined by any $\ML$-formula
of size less than the exponential tower of height  $n-1$. 

A similar gap between $\FO$ and temporal logic follows from a construction in the PhD thesis
\cite{stockmeyer} of Stockmeyer. He proved that the satisfiability problem of $\FO$
over words is of nonelementary complexity. Etessami and Wilke \cite{untilhierarchy}
observed that from Stockmeyer's proof it is possible to extract $\FO$-formulas of size 
$\ordo(n)$ whose smallest models are words of length nonelementary in $n$.  
On the other hand, it is well known that any satisfiable formula of temporal logic
has a model of size $\ordo(2^n)$, where $n$ is the size of the formula.


\section{Preliminaries}

In this section we fix some notation, define the syntax and semantics of basic modal logic and define our notions of formula size.
For a detailed account on the notions used in the paper, we refer to the textbook
\cite{blackburn} of Blackburn, de Rijke and Venema.

\subsection*{Basic modal logic and first-order logic}

Let $\M = (W, R, V)$, where $W$ is a set, $R \subseteq W \times W$ and $V : \Phi \to \powerset(W)$, and let $w \in W$. The structure $(\M, w)$ is called a \emph{pointed Kripke-model for $\Phi$}. 

Let $(\M, w)$ be a pointed Kripke-model. We use the notation
\[
\seur{\M}{w} := \{ (\M, v) \mid v \in M, wR^\M v\}.
\]
If $\AA$ is a set of pointed Kripke-models, we use the notation
\[
\kaikki\AA := \bigcup\limits_{(\M, w) \in \AA} \seur{\M}{w}.
\]
Furthermore, if $f$ is a function $f : \AA \to \kaikki\AA$ such that $f(\M, w) \in \seur{\M}{w}$ for every $(\M, w) \in \AA$, then we use the notation
\[
\di_f\AA := f(\AA).
\]

Now we define the syntax and semantics of basic modal logic for pointed models. 

\begin{definition}
Let $\Phi$ be a set of proposition symbols. The set of formulas of 
~$\ML(\Phi)$
is generated by the following grammar
\[
\varphi := p \mid \neg p \mid (\varphi \land \varphi ) \mid ( \varphi 
\lor \varphi )
\mid \Diamond \varphi \mid \Box \varphi,
\]
where $p \in \Phi$.
\end{definition}

As is apparent from the definition of the syntax, we assume that all $\ML$-formulas are in negation normal form. This is useful for the formula size game that we introduce in the next section.

\begin{definition}
The satisfaction relation $(\M, w) \vDash \phi$ between pointed Kripke-models $(\M, w)$
$\ML(\Phi)$-formulas $\phi$ is defined as follows:
\begin{enumerate}[(1)]
\item $(\M, w) \vDash p \Leftrightarrow w \in V(p)$,
\item $(\M, w) \vDash \neg p \Leftrightarrow w \notin V(p)$,
\item $(\M, w) \vDash (\phi \land \psi) \Leftrightarrow (\M, w) \vDash \phi$ and $(\M, w) \vDash \psi$,
\item $(\M, w) \vDash (\phi \lor \psi) \Leftrightarrow (\M, w) \vDash \phi$ or $(\M, w) \vDash \psi$,
\item $(\M, w) \vDash \di\phi \Leftrightarrow$ there is $(\M, v) \in \seur{\M}{w}$ such that $(\M, v) \vDash \phi$,
\item $(\M, w) \vDash \bo\phi \Leftrightarrow$ for every $(\M, v) \in \seur{\M}{w}$ it holds that $(\M, v) \vDash \phi$.
\end{enumerate}
Furthermore, if $\AA$ is a class of pointed Kripke-models, then 
\begin{enumerate}
\item[] $\AA \vDash \phi \Leftrightarrow (\A, w) \vDash \phi$ for every $(\A, w) \in \AA$.
\end{enumerate}
\end{definition}

In Section \ref{nelonen}, we also consider the case $\Phi=\emptyset$. For this purpose,
we add the atomic constants $\top$ and $\bot$ to $\ML$, where $(\M, w) \vDash \top$ and
$(\M, w) \nvDash \bot$ for all pointed Kripke models $(\M, w)$.

The syntax and semantics for first-order logic are defined in the standard way. 
Each $\ML$-formula $\phi$ defines a class $\MM(\phi)$ of pointed Kripke-models:
\[
	\MM(\phi):=\{(\M,w)\mid (\M, w) \vDash \phi\}.
\] 
In the same way, any $\FO$-formula $\psi(x)$ in the vocabulary consisting of the accessibility
relation symbol $R$ and unary relation symbols $U_p$ for $p\in\Phi$ defines
a class $\MM(\psi)$ of pointed Kripke-models:
\[
	\MM(\psi):=\{(\M,w)\mid \M \vDash \psi[w/x]\}.
\] 
The formulas $\phi\in\ML$ and $\psi(x)\in\FO$ are \emph{equivalent} if $\MM(\phi)=\MM(\psi)$.

The well-known link between $\ML$ and $\FO$ is the following theorem.

\begin{theorem}[van Benthem Characterization Theorem]
A first-order formula $\psi(x)$ is equivalent to some formula in $\ML$ if and only if $\MM(\psi)$ is bisimulation invariant.
\end{theorem}

If a property of pointed Kripke-models is $n$-bisimulation invariant for some $n\in\NN$, then it is also bisimulation invariant.  Thus, $\FO$-definability and $n$-bisimulation invariance imply 
$\ML$-definability for any property of pointed Kripke-models. We will use this version of 
van Benthem's characterization in Section \ref{sec-property} for showing that certain property
is $\ML$-definable.
For the sake of easier reading, we give here the definition of $n$-bisimulation.

\begin{definition}
Let $(\M, w)$ and $(\M', w')$ be pointed models. We say that $(\M, w)$ and $(\M', w')$ are \emph{$n$-bisimilar}, $(\M, w) \bisim_n (\M', w')$, if there
are binary relations $Z_n \subseteq \dots \subseteq Z_0$ such that for every $0 \leq i \leq n-1$ we have
\begin{enumerate}[(1)]
\item $(\M,w)Z_n(\M', w')$,
\item if $(\M, v)Z_0(\M',v')$, then $v$ and $v'$ are propositionally equivalent,
\item if $(\M,v) Z_{i+1} (\M',v')$ and $(\M, u) \in \seur{\M}{v}$ then there is $(\M', u') \in \seur{\M'}{v'}$ such that $(\M, u) Z_i (\M', u')$,
\item if $(\M,v) Z_{i+1} (\M',v')$ and $(\M', u') \in \seur{\M'}{v'}$ then there is $(\M, u) \in \seur{\M}{v}$ such that $(\M, u) Z_i (\M', u')$.
\end{enumerate}
\end{definition}

It is well known that two pointed
Kripke-models are $n$-bisimilar if and only if they are equivalent with respect to $\ML$-formulas of modal depth at most $n$.

\subsection*{Formula size}

We define notions of formula size for $\ML$ and $\FO$. These notions are related to the length of the formula as a string rather than the DAG-size\footnote{The DAG-size of a formula $\phi$
is the size of the syntactic structure of $\phi$ in the form of a DAG. Thus, it is less than 
$n^2$, where $n$ is the number of subformulas of $\phi$.}of it
For $\ML$ we define separately the number of modal operators and the number of binary connectives in the formula.

\begin{definition}
The \emph{modal size} of a formula $\phi \in \ML$, denoted $\ms(\phi)$, is defined recursively as follows:
\begin{enumerate}[(1)]
\item If $\phi$ is a literal, then $\ms(\phi) = 0$.
\item If $\phi = \psi \lor \vartheta$ or $\phi = \psi \land \vartheta$, then $\ms(\phi) = \ms(\psi) + \ms(\vartheta)$.
\item If $\phi = \di\psi$ or $\phi = \bo\psi$, then $\ms(\phi) = \ms(\psi) + 1$.
\end{enumerate}
\end{definition}

\begin{definition}
The \emph{binary connective size} of a formula $\phi \in \ML$, denoted by $\cs(\phi)$, is defined recursively as follows:
\begin{enumerate}[(1)]
\item If $\phi$ is a literal, then $\cs(\phi) = 0$.
\item If $\phi = \psi \lor \vartheta$ or $\phi = \psi \land \vartheta$, then $\cs(\phi) = \cs(\psi) + \cs(\vartheta)+1$.
\item If $\phi = \di\psi$ or $\phi = \bo\psi$, then $\cs(\phi) = \cs(\psi)$.
\end{enumerate}
\end{definition}

The size of an $\ML$ formula is defined as the sum of modal size and connective size. We do not count literals or parentheses since their number can be
derived from the number of binary connectives. 

\begin{definition}
The \emph{size} of a formula $\phi \in \ML$ is $\s(\phi) = \ms(\phi) + \cs(\phi)$.
\end{definition}

Similarly we define formula size for $\FO$ to be the number of binary connectives and quantifiers in the formula. In general this could lead to an arbitrarily
large difference between formula size and actual string length, but we only consider formulas with one binary relation so this is not an issue.

\begin{definition}
The \emph{size} of a formula $\phi \in \FO$, denoted by $\s(\phi)$, is defined recursively as follows:
\begin{enumerate}[(1)]
\item If $\phi$ is a literal, then $\s(\phi) = 0$.
\item If $\phi = \neg\psi$, then $\s(\phi) = \s(\psi)$.
\item If $\phi = \psi \lor \vartheta$ or $\phi = \psi \land \vartheta$, then $\s(\phi) = \s(\psi) + \s(\vartheta) + 1$.
\item If $\phi = \exists x\psi$ or $\phi = \forall x \psi$, then $\s(\phi) = \s(\psi) + 1$.
\end{enumerate}
\end{definition}

To refer to some rather large formula sizes we need the exponential tower function.

\begin{definition}
We define the function $\tower : \Nset \to \Nset$ recursively as follows:
\begin{align*}
\tower(0) &= 1 \\
\tower(n+1) &= 2^{\tower(n)}.
\end{align*}
We will also use in the sequel the binary logarithm function, denoted by $\log$.
\end{definition}

\subsection*{Separating classes by formulas}

The definition of the formula size game in the next section is based on the notion
of separating classes of pointed Kripke-models by formulas.

\begin{definition}
Let $\AA$ and $\BB$ be classes of pointed Kripke-models. 
\\
(a) We say that a formula $\phi \in \ML$ \emph{separates the classes $\AA$ and $\BB$} if 
$\AA \vDash \phi$ and $\BB \vDash \neg\phi$. 
\\
(b) Similarly, a formula $\psi(x) \in \FO$ separates the classes $\AA$ and $\BB$ if  
for all $(\M,w)\in \AA$, $ \M\vDash \psi[w/x]$ and for all $(\M,w)\in \BB$, $ \M\vDash \neg\psi[w/x]$.
\end{definition}

In other words, a formula $\phi\in\ML$ separates the classes $\AA$
and $\BB$ if $\AA\subseteq\MM(\phi)$ and $\BB\subseteq \overline{\MM(\phi)}$, where
$\overline{\MM(\phi)}$ is the complement of $\MM(\phi)$.

\section{The formula size game}

As in the Adler-Immerman game, the basic idea in our formula size game is that
there are two players, S (spoiler) and D (duplicator), who play on a pair $(\AA,\BB)$ of two
sets of pointed Kripke-models. The aim of S is to show that $\AA$ and $\BB$
can be separated by a formula with modal size at most $m$ and connective size
at most $k$, while D tries to refute this. The moves of S reflect the 
connectives and modal operators of a formula that is supposed to separate the sets. 

The crucial difference between our game and the Adler-Immerman game is that 
we define positions in the game to be tuples $(m,k,\AA,\BB)$ instead of just
pairs $(\AA,\BB)$. This means that in the connective moves, D has a genuine 
choice to make. Furthermore, the winning condition of the game is based
on a natural property of single positions instead of the size of the entire 
game tree. 

We give now the precise definition of our game.

\begin{definition}
Let $\AA_0$ and $\BB_0$ be sets of pointed Kripke-models and let $m_0,k_0\in\NN$. \emph{The formula size game between the sets $\AA_0$ and $\BB_0$}, denoted
$\EF{m_0,k_0}(\AA_0, \BB_0)$, has two players, S and D. The number $m_0$ is the \emph{modal parameter} and $k_0$ is the \emph{connective parameter} of the game. The
starting position of the game is $(m_0,k_0,\AA_0, \BB_0)$. Let the position after $n$ moves be $(m, k, \AA, \BB)$. To continue the game, S has the following four moves to
choose from:
\begin{itemize}
\item \emph{Left splitting move}: First, S chooses natural numbers $m_1$, $m_2$, $k_1$ and $k_2$ and sets $\AA_1$ and $\AA_2$ such that 
$m_1+m_2=m$, $k_1+k_2+1 = k$ and $\AA_1 \cup \AA_2 = \AA$. Then D decides whether the game continues from the position $(m_1,k_1, \AA_1, \BB)$ or the position 
$(m_2, k_2, \AA_2, \BB)$.

\item \emph{Right splitting move}: First, S chooses natural numbers $m_1$, $m_2$, $k_1$ and $k_2$ and sets $\BB_1$ and $\BB_2$ such that 
$m_1+m_2=m$, $k_1+k_2+1 = k$ and $\BB_1 \cup \BB_2 = \BB$. Then D decides whether the game continues from the position
$(m_1, k_1, \AA, \BB_1)$ or the position $(m_2, k_2, \AA, \BB_2)$.

\item \emph{Left successor move}: S chooses a function $f: \AA \to \kaikki\AA$ such that $f(\A, w) \in \seur{\A}{w}$ for all $(\A, w) \in \AA$
and the game continues from the position\\ $(m-1, k, \di_f\AA, \kaikki\BB)$. 

\item \emph{Right successor move}: S chooses a function $g: \BB \to \kaikki\BB$ such that $g(\B, w) \in \seur{\B}{w}$ for all $(\B, w) \in
\BB$ and the game continues from the position\\ $(m-1, k, \kaikki\AA, \di_g\BB)$. 
\end{itemize}
The game ends and S wins in a position $(m, k, \AA, \BB)$ if there is a literal $\phi$ such that $\phi$ separates the sets $\AA$ and
$\BB$.
The game ends and D wins in a position $(m, k, \AA, \BB)$ if $m = k = 0$ and S does not win in this position.
\end{definition}

Note that if $\seur{\M}{w} = \emptyset$ for some $(\M, w) \in \AA$ ($\in\BB$) then S cannot make a left (right) successor move in the position $(m,k,\AA,\BB)$.

We prove now that the formula size game indeed characterizes the separation of 
two sets of pointed Kripke-models by a formula of a given size.

\begin{theorem}\label{peruslause}
Let $\AA$ and $\BB$ be sets of pointed models and let $m$ and $k$ be natural numbers. Then the following conditions are equivalent:
\begin{enumerate}[11111111]
\item[$\win_{m,k}$] S has a winning strategy in the game $\EF{m,k}(\AA, \BB)$.
\item[$\sep_{m,k}$] There is a formula $\phi \in \ML$ such that $\ms(\phi) \le m$, $\cs(\phi) \le k$ and the formula $\phi$ separates the sets $\AA$ and $\BB$.
\end{enumerate}
\end{theorem}
\begin{proof}
The proof proceeds by induction on the number $m+k$. If $m+k = 0$, no moves can be made. Thus if S wins, then there is a literal $\phi$ that
separates the sets $\AA$ and $\BB$. In this case $\s(\phi) = 0$ so $\win_{0,0} \Rightarrow \sep_{0,0}$. On the other hand, if there is a formula $\phi$ such
that $\s(\phi) \le 0$ and $\phi$ separates the sets $\AA$ and $\BB$, then $\phi$ is a literal. Thus S wins the game, and we see that $\sep_{0,0} \Rightarrow \win_{0,0}$.

Suppose then that $m+k > 0$ and $\win_{n,l} \Leftrightarrow \sep_{n,l}$ for all $n,l \in \Zset_+$ such that $n+l < m+k$. Assume first that $\win_{m,k}$ holds.
Consider the following cases according to the first move in the winning strategy of S.
\begin{enumerate}[(a)]
\item Assume that the first move of the winning strategy of S is a left splitting move choosing numbers $m_1, m_2, k_1, k_2 \in \Nset$ such that $m_1+m_2 = m$ and
$k_1+k_2+1=k$, and sets $\AA_1, \AA_2 \subseteq \AA$ such that $\AA_1 \cup \AA_2 = \AA$. Since this move is given by a winning strategy, S has a
winning strategy for both possible continuations of the game, $(m_1, k_1, \AA_1, \BB)$ and $(m_2, k_2, \AA_2, \BB)$. Since $m_i + k_i < m_i+k_i + 1 \le m+k$
for $i \in \{1,2\}$, by induction hypothesis there is a formula $\psi$ such that $\ms(\psi) \le m_1$, $\cs(\psi) \le k_1$ and $\psi$ separates the sets
$\AA_1$ and $\BB$ and a formula $\vartheta$ such that $\ms(\vartheta) \le m_2$, $\cs(\vartheta) \le k_2$ and $\vartheta$ separates the sets $\AA_2$ and
$\BB$. Thus $\AA_1 \vDash \psi$ and $\AA_2 \vDash \vartheta$ so $\AA \vDash \psi \lor \vartheta$. On the other hand $\BB \vDash \neg\psi$ and $\BB \vDash
\neg\vartheta$ so $\BB \vDash \neg(\psi \lor \vartheta)$. Therefore the formula $\psi \lor \vartheta$ separates the sets $\AA$ and $\BB$. In addition
$\ms(\psi \lor \vartheta) = \ms(\psi) + \ms(\vartheta) \le m_1+m_2 = m$ and $\cs(\psi \lor \vartheta) = \cs(\psi) + \cs(\vartheta) + 1 \le k_1+k_2+1 = k$ so
$\sep_{m,k}$ holds.

\item The case in which the first move of the winning strategy of S is a right splitting move
is proved in the same way as the previous one, with the roles of $\AA$ and $\BB$
switched, and disjunction replaced by conjunction.

\item Assume that the first move of the winning strategy of S is a left successor move choosing a function $f: \AA \to \kaikki\AA$ such that $f(\A, w) \in \seur{\A}{w}$ for all
 $(\A, w) \in \AA$. The game continues from the position $(m-1, k, \di_f\AA, \kaikki\BB)$ and S has a winning strategy from this position. By induction 
hypothesis there is a formula $\psi$ such that $\ms(\psi) \le m-1$, $\cs(\psi) \le k$ and $\psi$ separates the sets $\di_f\AA$ and $\kaikki\BB$. Now for every $(\A, w)
\in \AA$ we have $f(\A, w) \in \seur{\A}{w}$ and $f(\A, w) \vDash \psi$. Therefore $\AA \vDash \di\psi$. On the other hand $\kaikki\BB \vDash \neg\psi$ so for every $(\B, w) \in \BB$ and every 
$(\B, v) \in \seur{\B}{w}$ we have $(\B, v) \vDash \neg\psi$. Therefore $\BB \vDash \bo\neg\psi$ and thus $\BB \vDash \neg\di\psi$. So the formula $\di\psi$ 
separates the sets $\AA$ and $\BB$ and since $\ms(\di\psi) = \ms(\psi) + 1 \le m$ and $\cs(\di\psi) = \cs(\psi) \le k$, $\sep_{m,k}$ holds.

\item The case in which the first move of the winning strategy of S is a right successor move
is similar to the case of left successor move. It suffices to switch the classes $\AA$ and $\BB$,
and replace $\di$ with $\bo$. 

\end{enumerate}

Now assume $\sep_{m,k}$ holds, and $\phi$ is the formula separating $\AA$ and $\BB$. 
We obtain a winning strategy of S 
for the game $\EF{m,k}(\AA, \BB)$ using 
$\phi$ as follows:
\begin{enumerate}[(a)]

\item If $\phi$ is a literal, S wins the game with no moves.

\item Assume that $\phi = \psi \lor \vartheta$. Let $\AA_1 := \{(\A, w) \in \AA \mid (\A, w) \vDash \psi\}$ and $\AA_2 := \{(\A, w) \in \AA \mid (\A, w) \vDash
\vartheta\}$. Since $\AA \vDash \phi$ we have $\AA_1 \cup \AA_2 = \AA$. In addition, since $\BB \vDash \neg\phi$, we have $\BB \vDash \neg\psi$ and 
$\BB \vDash \neg\vartheta$. Thus $\psi$ separates the sets $\AA_1$ and $\BB$ and $\vartheta$ separates the sets $\AA_2$ and $\BB$. Since $\ms(\psi) 
+ \ms(\vartheta) = \ms(\phi) \le m$, there are $m_1, m_2 \in \Nset$ such that $m_1+m_2 = m$, $\ms(\psi) \le m_1$ and $\ms(\vartheta) \le m_2$. Similarly
since $\cs(\psi)+\cs(\vartheta)+1=\cs(\phi) \le k$, there are $k_1, k_2 \in \Nset$ such that $k_1+k_2+1 = k$, $\cs(\psi) \le k_1$ and $\cs(\vartheta) \le k_2$.
By induction hypothesis S has winning strategies for the games $\EF{m_1,k_1}(\AA_1, \BB)$ and $\EF{m_2,k_2}(\AA_2, \BB)$. Since $k \ge \cs(\phi)
\ge 1$, S can start the game $\EF{m,k}(\AA, \BB)$ with a left splitting move choosing the numbers $m_1$, $m_2$, $k_1$ and $k_2$ and the sets $\AA_1$
and $\AA_2$. Then S wins the game by following the winning strategy for 
whichever position D chooses.

\item The case $\phi = \psi \land \vartheta$ is similar to the case of disjunction. This time
S uses the sets $\BB_1 := \{(\B, w) \in \BB \mid (\B, w) \vDash \neg\psi\}$ and 
$\BB_2 := \{(\B, w) \in \BB \mid (\B, w) \vDash\neg\vartheta\}$ for choosing a right splitting move.

\item Assume that $\phi = \di\psi$. Since $\AA \vDash \phi$, for every $(\A, w) \in \AA$ there is $(\A, v_w) \in \seur{\A}{w}$ such that $(\A, v_w) \vDash \psi$. 
We define the function $f : \AA \to \kaikki\AA$ by $f(\A, w) = (\A, v_w)$. Clearly $\di_f\AA \vDash \psi$. On the other hand $\BB \vDash \neg \phi$ so $\BB \vDash 
\bo\neg\psi$ and thus for each $(\B, w) \in \BB$ and each $(\B, v) \in \seur{\B}{w}$ we have $(\B, v) \vDash \neg\psi$. Therefore $\kaikki\BB \vDash \neg\psi$
and the formula $\psi$ separates the sets $\di_f\AA$ and $\kaikki\BB$. Moreover,  $\ms(\psi) = \ms(\phi) - 1 \le m - 1$ and $\cs(\psi) = \cs(\phi) \le k$ so by
induction hypothesis S has a winning strategy for the game $\EF{m-1,k}(\di_f\AA, \kaikki\BB)$. Since $m \ge \ms(\phi) \ge 1$, S can start the game 
$\EF{m,k}(\AA, \BB)$ with a left successor move choosing the function $f$. Then S wins the game by following the winning strategy for the game 
$\EF{m-1,k}(\di_f\AA, \kaikki\BB)$.

\item The case $\phi = \bo\psi$ is proved in the same way as the case with $\di$.
Again it suffices to switch $\AA$ and $\BB$, and use right successor move instead of
left successor move. 
\end{enumerate}
\end{proof}

We prove next that $m$-bisimilarity implies that D has winning strategy in the 
formula size game with modal parameter $m$. 

\begin{theorem}\label{bisim}
Let $\AA$ and $\BB$ be sets of pointed models and let $m,k \in \Nset$. If there are $m$-bisimilar pointed models $(\A, w) \in \AA$ and $(\B, v) \in \BB$, 
then D has a winning strategy for the game $\EF{m,k}(\AA, \BB)$.
\end{theorem}
\begin{proof}
The proof proceeds by induction on the number $m+k \in \Nset$. If $m+k = 0$ and $(\A, w) \in \AA$ and $(\B, v) \in \BB$ are $m$-bisimilar, then they are 
$0$-bisimilar and thus satisfy the same literals. Thus there is no literal $\phi \in \ML$ that separates the sets $\AA$ and $\BB$. Since S cannot make any
moves and S does not win the game in this position, D wins the game $\EF{0,0}(\AA, \BB)$.

Assume that $m+k > 0$ and $(\A, w) \in \AA$ and $(\B, v) \in \BB$ are $m$-bisimilar. As in the basic step, S does not win the game in this position. We
consider the cases of the first move of S in the game $\EF{m,k}(\AA, \BB)$.

If S starts with a left splitting move choosing the numbers $m_1$, $m_2$, $k_1$ and $k_2$ and the sets $\AA_1$ and $\AA_2$, then since $\AA_1 \cup \AA_2
= \AA$, D can choose the next position $(m_i, k_i, \AA_i, \BB)$, $i \in \{1, 2\}$ in such a way that $(\A, w) \in \AA_i$. Then we have
$m_i \le m$ and $m_i+k_i < m+k$ so by induction hypothesis D has a winning strategy for the game $\EF{m_i,k_i}(\AA_i, \BB)$. The case of a right splitting
move is similar.

If S starts with a left successor move choosing a function $f : \AA \to \kaikki\AA$, then since $(\A, w)$ and $(\B, v)$ are $m$-bisimilar, there is a pointed
model $(\B, v') \in \seur{\B}{v}$ that is $m-1$-bisimilar with the pointed model $f(\A, w)$. Since $m-1+k < m+k$, by induction hypothesis D has a winning
strategy in $\EF{m-1,k}(\di_f\AA, \kaikki\BB)$. The case of a right successor move is similar.
\end{proof}

\section{Succinctness of FO over ML}\label{nelonen}

In this section, we illustrate the use of the formula size game $\EF{m,k}$ by proving
a nonelementary succinctness gap between bisimulation invariant first-order logic
and modal logic.

\subsection{A property of pointed frames}\label{sec-property}

For the remainder of this paper we consider only the case where the set $\Phi$ of propositional symbols is empty. This makes all points in Kripke-models propositionally
equivalent so we call pointed models in this section pointed frames. 
The only formulas available for the win condition of S in the game $\EF{m,k}$ are
$\bot$ and $\top$. Thus S only wins the game from the position $(m,k,\AA, \BB)$ if either $\AA = \emptyset$ and $\BB \neq \emptyset$ or 
$\AA \neq \emptyset$ and $\BB = \emptyset$.

We will use the following two classes in our application of the formula size game $\EF{m,k}$:
\begin{itemize}
\item $\AA_n$ is the class of all pointed frames $(\A, w)$ such that for all
 $(\A, u), (\A, v) \in \seur{\A}{w}$, the frames $(\A, u)$ and $(\A, v)$ are $n$-bisimilar.
\item $\BB_n$ is the complement of $\AA_n$.
\end{itemize}

\begin{lemma}\label{FO}
For each $n \in \Nset$ there is a formula $\phi_n(x) \in \FO$ that separates the classes $\AA_n$ and $\BB_n$ such that the size of $\phi_n(x)$ is exponential with 
respect to $n$, i.e., $s(\phi_n) = \ordo(2^n)$.
\end{lemma}
\begin{proof}
We first define formulas $\psi_n(x,y) \in \FO$ such that 
$(\M, u)\bisim_n (\M, v)$ if and only if $\M \vDash \psi_n[u/x, v/y]$.
The formulas $\psi_n(x,y)$ are defined recursively as follows:
\begin{align*}
\psi_1(x,y) := &\exists s R(x,s) \leftrightarrow \exists t R(y,t) \\
\psi_{n+1}(x,y) := &\forall s(R(x,s) \rightarrow \exists t(R(y,t) \land \psi_n(s,t)) \\
			& \land \forall t(R(y,t) \rightarrow \exists s(R(x,s) \land \psi_n(s,t)).
\end{align*}
Clearly these formulas express $n$-bisimilarity as intended. When we interpret the equivalences and implications as shorthand in the standard way,
we get the sizes $\s(\psi_1) = 11$ and $\s(\psi_{n+1}) = 2 \cdot \s(\psi_n) + 13$. Thus $\s(\psi_n) = 3 \cdot 2^{n+2} - 13$.

Now we can define the formulas $\phi_n$:
\[
\phi_n(x) := \forall y \forall z (R(x,y) \land R(x,z) \rightarrow \psi_n(y,z)).
\]
Clearly for every $(\A, w) \in \AA_n$ we have $\A \vDash \phi_n[w/x]$ and for every $(\B, v) \in \BB_n$ we have $\B \vDash \neg \phi_n[w/x]$ so the formula
$\phi_n$ separates the classes $\AA_n$ and $\BB_n$. Furthermore, $\s(\phi_n) = \s(\psi_n) + 6 = 3 \cdot 2^{n+2} - 7$ so the size of $\phi_n$ is exponential with
respect to $n$.
\end{proof}

\begin{lemma}
For each $n \in \Nset$, the formula $\phi_n$ is $n+1$-bisimulation invariant.
\end{lemma}
\begin{proof}
Let $(\A, w)$ and $(\B, v)$ be $n+1$-bisimilar pointed models. Assume that $\A \vDash \phi_n[w/x]$. If $(\B, v_1), (\B, v_2) \in \seur{\B}{v}$, by $n+1$-bisimilarity there are 
$(\A, w_1), (\A, w_2) \in \seur{\A}{w}$ such that $(\A, w_1) \bisim_n (\B, v_1)$ and $(\A, w_2) \bisim_n (\B, v_2)$. Since $\A \vDash \phi_n[w/x]$, we have 
$(\B, v_1) \bisim_n (\A, w_1) \bisim_n (\A, w_2) \bisim_n (\B, v_2)$ so $\B \vDash \psi_n[v_1/x, v_2/y]$. Thus, we see that $\B \vDash \phi_n[v/x]$.
\end{proof}

It follows now from van Benthem's characterization theorem that each $\phi_n$ is equivalent
to some $\ML$-formula. Thus, we get the following corollary.

\begin{corollary}\label{Rosen}
For each $n \in \Nset$, there is a formula $\vartheta_n \in \ML$ that separates the classes $\AA_n$ and $\BB_n$.
\end{corollary}

\subsection{Set theoretic construction of pointed frames}

We have shown that the classes $\AA_n$ and $\BB_n$ can be separated both in $\ML$ and in $\FO$. Furthermore the size of the FO-formula is exponential with respect to
$n$. It only remains to ask: what is the size of the smallest $\ML$-formula that separates the classes $\AA_n$ and $\BB_n$? To answer this we will need suitable subsets of $\AA_n$ and $\BB_n$ to play the formula size game
on.

\begin{definition}
Let $n \in \Nset$. \emph{The finite levels of the cumulative hierarchy} are defined recursively as follows:
\begin{align*}
V_0 &= \emptyset \\
V_{n+1} &= \powerset(V_n)
\end{align*}
\end{definition}

For every $n \in \Nset$, $V_n$ is a \emph{transitive set}, i.e., for every $a \in V_n$ and every $b \in a$ it holds that $b \in V_n$. Thus it is reasonable to define a frame $\F_n = (V_n, R_n)$, where for all $a, b \in V_n$ it holds that $(a, b) \in R_n \Leftrightarrow b \in a$.

For every set $a \in V_n$ we define a pointed frame $(\M_a, a)$, where $\M_a$ is the subframe of $\F_n$ generated by the point $a$. 

\begin{lemma}\label{notbisim}
Let $n \in \Nset$ and $a, b \in V_{n+1}$. If $a \neq b$, then $(\M_a, a)\not\bisim_n(\M_b, b)$. 
\end{lemma}
\begin{proof}
We prove the claim by induction on $n$. The basic step $n = 0$ is trivial since $V_1$ only has one element. For the induction step, assume that 
$a, b \in V_{n+1}$ and $a \neq b$. Assume further for contradiction that 
$(\M_a, a)\bisim_n (\M_b, b)$. Since $a \neq b$, by symmetry we can 
assume that there is $x \in a$ such that $x \notin b$. By $n$-bisimilarity there is $y \in b$ such that $(\M_x, x)$ and $(\M_y, y)$ are $n-1$-bisimilar. Since 
$x \in a \in V_{n+1}$ and $y \in b \in V_{n+1}$, we have $x ,y \in V_n$. By induction hypothesis we obtain $x = y$. This is a contradiction, since $x \notin b$ 
and $y \in b$.
\end{proof}

Let $\AA$ be a set of pointed frames and $w \notin \dom(\A)$ for all $(\A, v) \in \AA$. We use the notation $\liitos\AA := (\M, w)$, where 
\begin{align*}
&\dom(\M) = \{w\} \cup \bigcup\{\dom(\A) \mid (\A, v) \in \AA\}, \text{ and }\\
 &R^\M = \{(w,v) \mid (\A, v) \in \AA \} \cup \bigcup\{R^\A \mid (\A, v) \in \AA\}.
\end{align*}

For $n \in \Nset$ we further use the notation
\begin{align*}
\VV_n &:= \{ \liitos \{(\M_a, a)\} \mid a \in V_{n+1} \} \\
\EE_n &:= \{ \liitos \{ (\M_a, a), (\M_b, b) \} \mid a, b \in V_{n+1}, a \neq b \}
\end{align*}

It is well known that the cardinality of $V_n$ is the exponential tower of $n-1$, 
Thus, the cardinality of $\VV_n$ is $\tower(n)$.

\begin{lemma}\label{tower}
If $n \in \Nset$, we have $|\VV_n| = |V_{n+1}| = \tower(n)$. \qed
\end{lemma}

\subsection{Graph colorings and winning strategies in $\EF{m,k}$}

Our aim is to prove that any $\ML$-formula $\theta_n$ separating the
sets $\VV_n$ and $\EE_n$ is of size at least $\tower(n-1)$. To do this,
we make use of a surprising connection between the chromatic numbers
of certain graphs related to pairs of the form $(\VV,\EE)$, and existence
of a winning strategies for D in the game $\EF{m,k}(\VV,\EE)$.

Let $n \in \Nset$, $\emptyset \neq \VV \subseteq \VV_n$ and $\EE \subseteq \EE_n$. Then $\G (\VV, \EE)$ denotes the graph $(V, E)$, where
\begin{align*}
V &= \{ (\M, w) \mid \liitos \{(\M, w) \} \in \VV \}, \text{ and }\\
 E &= \{ ((\M, w), (\M', w')) \in V \times V \mid \liitos \{(\M, w), (\M', w')\} \in \EE\}.
\end{align*}

\begin{definition}
Let $\G = (V, E)$ be a graph and let $C$ be a set. A function $\chi : V \to C$ is a \emph{coloring} of the graph $\G$ if for all $u, v \in V$ it holds that 
if $(u, v) \in E$, then $\chi(u) \neq \chi(v)$. If the set $C$ has $k$ elements, then $\chi$ is called a \emph{$k$-coloring} of $\G$.

The \emph{chromatic number} of $\G$, denoted by $\chi(\G)$, is the smallest number $k \in \Nset$ for which there is a $k$-coloring of $\G$.
\end{definition}

\begin{lemma}\label{väritys}
Let $\G = (V, E)$ be a graph. 
\begin{enumerate}
\item Let $V_1, V_2 \subseteq V$ be nonempty such that $V_1 \cup V_2 = V$ and let $\G_1 = (V_1, E \upharpoonright V_1)$ and
$\G_2 = (V_2, E \upharpoonright V_2)$. Then we have  $\chi (\G) \le \chi (\G_1) + \chi (\G_2)$.
\item Let $E_1, E_2 \subseteq E$ such that $E_1 \cup E_2 = E$ and let  $\G_1 = (V, E_1)$ and \mbox{$\G_2 = (V, E_2)$}. Then $\chi (\G) \le \chi (\G_1)\chi (\G_2)$.
\end{enumerate}
\end{lemma}
\begin{proof}
(i) Let $V_1$, $V_2$, $\G_1$ and $\G_2$ be as in the claim and let $k_1 = \chi(\G_1)$ and $k_2 = \chi(\G_2)$. Let $\chi_1 : V_1 \to \{1, \dots, k_1\}$ 
be a $k_1$-coloring of the graph $\G_1$ and let $\chi_2 : V_2 \to \{k_1+1, \dots, k_1+k_2 \}$ be a $k_2$-coloring of the graph $\G_2$. Then it is straightforward to show that
$\chi = \chi_1 \cup (\chi_2 \upharpoonright (V_2 \setminus V_1))$ is a $k_1+k_2$-coloring of the graph $\G$, whence $\chi (\G) \le k_1+k_2 =\chi (\G_1) + \chi (\G_2)$.

(ii) Let $\chi_1 : V \to \{1, \dots, k_1\}$ and $\chi_2 : V \to \{1, \dots, k_2\}$ be colorings of the graphs $\G_1$ and $\G_2$, respectively. Then it is easy to verify that the map 
$\chi : V \to \{1, \dots, k_1\} \times \{1, \dots, k_2\}$ defined by $\chi (v) = (\chi_1(v), \chi_2(v))$ is a coloring of $\G$. Thus we obtain $\chi (\G) \le |\{1, \dots, k_1\} \times 
\{1, \dots, k_2\}| = \chi (\G_1) \chi (\G_2)$.
\end{proof}

\begin{lemma}\label{winstrat}
Assume $\emptyset \neq \VV \subseteq \VV_n$ and $\EE \subseteq \EE_n$ for some $n \in \Nset$ and let $m, k \in \Nset$. If $\chi (\G(\VV, \EE)) \ge 2$ and 
$k < \log ( \chi (\G(\VV, \EE)))$, then D has a winning strategy in the game $\EF{m,k}(\VV, \EE)$.
\end{lemma}
\begin{proof}
Let $n,m,k \in \Nset$ and assume that $\emptyset \neq \VV \subseteq \VV_n$, $\EE \subseteq \EE_n$, $\chi (\G (\VV, \EE)) \ge 2$ and $k < \log (\chi (\G (\VV, \EE)))$.
We prove the claim by induction on $k$.

If $k = 0$, S can only make successor moves. Since $\chi (\G (\VV, \EE)) \ge 2$, there are $(\M, w), (\M', w') \in V$ such that $((\M, w), (\M', w')) \in E$.
Thus $\liitos \{(\M, w)\}$, $\liitos \{(\M', w')\} \in \VV$ and $\liitos \{(\M, w), (\M',w')\} \in \EE$. If S makes a left or right successor move, then in the 
resulting position $(m-1, 0, \VV', \EE')$ it holds that $(\M, w) \in \VV' \cap \EE'$ or $(\M', w') \in \VV' \cap \EE'$. Thus the same pointed model is present on both
sides of the game and by Theorem \ref{bisim}, D has a winning strategy for the game $\EF{m,k}(\VV', \EE')$.

Assume then that $k > 0$. If S starts the game with a successor move, then D wins as described above.

Assume that S begins the game with a left splitting move choosing the numbers $m_1,m_2,k_1, k_2 \in \Nset$ and the sets $\VV_1, \VV_2 \subseteq \VV$. Consider
the graphs $\G (\VV, \EE) = (V, E)$, $\G(\VV_1, \EE) = (V_1, E_1)$ and $\G(\VV_2, \EE) = (V_2, E_2)$. Since $\VV_1 \cup \VV_2 = \VV$, we have $V_1 \cup
V_2 = V$. In addition, by the definition of the graphs $\G(\VV, \EE)$, $\G(\VV_1, \EE)$ and $\G(\VV_2, \EE)$ we see that $E_1 = E \upharpoonright V_1$ and $E_2 = E \upharpoonright V_2$. 
Thus by Lemma \ref{väritys}, we obtain $\chi(\G(\VV, \EE)) \le \chi(\G(\VV_1, \EE)) + \chi(\G(\VV_2, \EE))$. It must hold that $k_1 < \log(\chi(\G(\VV_1, \EE)))$ or
$k_2 < \log(\chi(\G(\VV_1, \EE)))$, since otherwise we would have
\begin{align*}
k &< \log(\chi(\G(\VV, \EE))) \le \log(\chi(\G(\VV_1, \EE)) + \chi(\G(\VV_2, \EE)))  \\ 
& \le \log(\chi(\G(\VV_1, \EE))) + \log(\chi(\G(\VV_2, \EE))) + 1 \le k_1 + k_2 + 1 = k.
\end{align*}
Thus D can choose the next position of the game, $(m_i, k_i, \VV_i, \EE)$, in such a way that 
$k_i < \log(\chi(\G(\VV_i, \EE)))$. By induction hypothesis
D has a winning strategy in the game $\EF{m_i,k_i}(\VV_i, \EE)$. 

Assume then that S begins the game with a right splitting move choosing the numbers $m_1,m_2,k_1, k_2 \in \Nset$ and the sets $\EE_1, \EE_2 \subseteq \EE$. Consider
now the graphs $\G(\VV, \EE)=(V, E)$, $\G(\VV, \EE_1)=(V_1, E_1)$ and $\G(\VV, \EE_2)=(V, E_2)$. Clearly $V_1 = V_2 = V$ and since 
$\EE_1 \cup \EE_2 = \EE$, we have $E_1 \cup E_2 = E$. Thus by Lemma~\ref{väritys}, we obtain  $\chi(\G(\VV, \EE)) \le \chi(\G(\VV, \EE_1))\chi(\G(\VV, \EE_2))$.
It must hold that $k_1 < \log(\chi(\G(\VV_1, \EE)))$ or $k_2 < \log(\chi(\G(\VV_1, \EE)))$, since otherwise we would have
\begin{align*}
k &< \log(\chi(\G(\VV, \EE))) \le \log(\chi(\G(\VV, \EE_1))\chi(\G(\VV, \EE_2))) \\
& = \log(\chi(\G(\VV, \EE_1))) + \log(\chi(\G(\VV, \EE_2))) \le k_1+k_2+1 = k.
\end{align*}
Thus D can again choose the next position of the game, $(m_i, k_i, \VV, \EE_i)$, in such a way that $k_i < \log(\chi(\G(\VV_i, \EE)))$. By induction hypothesis
D has a winning strategy in the game $\EF{m_i, k_i}(\VV, \EE_i)$.
\end{proof}

\begin{lemma}\label{winstrat2}
If $k < \tower(n-1)$ and $m \in \Nset$, then D has a winning strategy in the game $\EF{m,k}(\VV_n, \EE_n)$.
\end{lemma}
\begin{proof}
By Lemma \ref{tower}, we have $|\VV_n| = \tower(n)$ and the set $\EE_n$ consists of all the pointed frames $\liitos\{(\M, w), (\M', w')\}$, where $(\M, w),
(\M', w') \in \VV_n$, $(\M, w) \neq (\M', w')$. Thus the graph $\G(\VV_n, \EE_n)$ is isomorphic with the complete graph $K_{\tower(n)}$. Therefore we obtain
\[
\chi(\G(\VV_n, \EE_n)) = \chi(K_{\tower(n)}) = \tower(n).
\]
By the assumption, $k < \tower(n-1) = \log(\tower(n)) = \log(\chi(\G(\VV_n, \EE_n)))$, so by Lemma \ref{winstrat}, D has a winning strategy in the game 
$\EF{m,k}(\VV_n, \EE_n)$.
\end{proof}

\begin{theorem}\label{ML}
Let $n \in \Nset$. If a formula $\theta_n \in \ML$ separates the classes $\AA_n$ and $\BB_n$, then $s(\theta_n) \ge \tower(n-1)$.
\end{theorem}
\begin{proof}
Assume that a formula $\theta_n \in \ML$ separates the classes $\AA_n$ and $\BB_n$. Since for every pointed frame $(M, w) \in \VV_n$ the set $\seur{\M}{w}$ 
has only one element, we have $\VV_n \subseteq \AA_n$. On the other hand, every pointed frame in the set $\EE_n$ is of the form $\liitos\{(\M_a, a), (\M_b, b)\}$,
where $a, b \in V_{n+1}$, $a \neq b$. By Lemma~\ref{notbisim},  $(\M_a, a)\not\bisim_n(\M_b, b)$ so $\EE_n \subseteq \BB_n$.

Assume for contradiction that $\s(\theta_n) < \tower(n-1)$. By Theorem \ref{peruslause}, S has a winning strategy in the game $\EF{m,k}(\VV_n, \EE_n)$ for $m = \ms(\theta_n)$ and $k = \cs(\theta_n)$. On the other hand, $k < \tower(n-1)$, whence by Lemma \ref{winstrat2}, D has a winning strategy in the same game.
\end{proof}

We now have everything we need for proving the nonelementary succinctness of $\FO$ 
over $\ML$. By Lemma \ref{FO}, for each $n \in \Nset$ there is a formula $\phi_n(x) \in \FO$ such that $\phi_n$ separates the classes $\AA_n$ and $\BB_n$ with $s(\phi)=\ordo(2^n)$. 
On the other hand by Corollary \ref{Rosen}, there is an equivalent formula $\vartheta_n \in \ML$, but by Theorem \ref{ML} the size of $\vartheta_n$ 
must be at least $\tower(n-1)$. 

\begin{corollary}
Bisimulation invariant $\FO$ is  nonelementarily more succinct than $\ML$.
\end{corollary}


\bibliographystyle{aiml16}
\bibliography{aiml16}

\end{document}